\documentclass[review,fleqn]{elsarticle}

\usepackage{lineno,hyperref}
\modulolinenumbers[5]

\journal{Journal of Mathematical Analysis and Applications}

\usepackage[margin=1in]{geometry}
\usepackage{amsmath,amsfonts,graphicx,framed}
\usepackage{amssymb}
\usepackage{amsthm}
\usepackage{fancyhdr}
\usepackage{float}
\usepackage{caption}
\usepackage{comment}
\usepackage{dsfont}
\usepackage{cleveref}
\crefname{section}{\S}{\S\S}
\Crefname{section}{\S}{\S\S}

\theoremstyle{plain}% default
\newtheorem{thm}{Theorem}
\newtheorem{lem}{Lemma}
\newtheorem{cor}{Corollary}

\theoremstyle{definition}
\newtheorem{defn}{Definition}

\theoremstyle{remark}
\newtheorem{rem}{Remark}

\numberwithin{equation}{section}
\numberwithin{prop}{section}
\numberwithin{thm}{section}
\numberwithin{lem}{section}
\numberwithin{cor}{section}
\numberwithin{rem}{section}
\numberwithin{defn}{section}

\begin{document}

\begin{frontmatter}

\title{On a nonlocal extension of differentiation}

\author[mymainaddress]{Ravi Shankar}
\address{Department of Mathematics, University of Nebraska-Lincoln, 1400 R Street, Lincoln, NE 68588, United States}
\ead{rshankar@mail.csuchico.edu}
\address[mymainaddress]{Department of Mathematics and Statistics, CSU Chico, 400 West First Street Chico, California 95929, United States}

\begin{abstract}
We study an integral equation that extends the problem of anti-differentiation.  We formulate this equation by replacing the classical derivative with a known nonlocal operator similar to those applied in fracture mechanics and nonlocal diffusion.  We show that this operator converges weakly to the classical derivative as a nonlocality parameter vanishes.  Using Fourier transforms, we find the general solution to the integral equation.  We show that the nonlocal antiderivative involves an infinite dimensional set of functions in addition to an arbitrary constant.  However, these functions converge weakly to zero as the nonlocality parameter vanishes.  For special types of integral kernels, we show that the nonlocal antiderivative weakly converges to its classical counterpart as the nonlocality parameter vanishes.
\end{abstract}

\begin{keyword}
nonlocal operators\sep integral equations\sep Fourier transform\sep tempered distributions\sep peridynamics\sep nonlocal diffusion
\end{keyword}

\end{frontmatter}

%\linenumbers

\section{Introduction}
We consider integral equations for distributions $u_\epsilon$ on $\mathbb{R}$ of the form:
\begin{align}\label{form}
D_{\alpha,\epsilon}u_\epsilon(t):=-\alpha_\epsilon\ast u_\epsilon(t)=F(t),\hspace{4 mm}t\in\mathbb{R},
\end{align}
where $\ast$ denotes the convolution, $f\ast g\,(t)=\int_{\mathbb{R}}f(t-s)g(s)\,\mathrm ds$, $\alpha_\epsilon$ is an anti-symmetric function on $\mathbb{R}$ that depends on a positive ``nonlocality parameter" $\epsilon$, and $D_{\alpha,\epsilon}$ is a ``nonlocal derivative".  

Nonlocality describes interactions over distances; in \eqref{form}, it refers to the fact that $F(t)$ is related to $u(s)$ via $D$ for values of $s$ far from $t$, where ``far" is quantified by the parameter $\epsilon$.  Mathematically speaking, $D_{\alpha_\epsilon}$ is ``strongly nonlocal" in the sense of Rogula \cite{Rogula}, since the support of $D_{\alpha,\epsilon}u$ is not contained in that of $u$.  In contrast, the derivative $u'(t)$ does satisfy this property, so it is a ``local" operator (it is also ``weakly nonlocal"; see \cite{Rogula}).

We define $D_{\alpha,\epsilon}$ in such a way that $D_{\alpha,\epsilon}f(t)\to f'(t)$ in some sense as $\epsilon\to 0$.  In other words, we think of (\ref{form}) as a ``nonlocal extension" of the first order ordinary differential equation (ODE):
\begin{align}\label{ode}
u'(t)=F(t),\hspace{4 mm}t\in\mathbb{R}.
\end{align}
We solve \eqref{form} using the Fourier transform:
\begin{align}\label{fourier}
\begin{split}
&\hat{u}(\xi)=\mathcal{F}[u](\xi)=\int_{-\infty}^\infty e^{-2\pi i\xi t}u(t)\mathrm dt,\\
&u(t)=\mathcal{F}^{-1}[\hat{u}](t)=\int_{-\infty}^\infty e^{2\pi it\xi}\hat{u}(\xi)\mathrm d\xi.
\end{split}
\end{align}

Nonlocal models, though well known since the 1800's viz. fractional derivatives, have recently found successful application in Silling's \cite{Silling2000} theory of peridynamic fractures; for applications of these nonlocal operators in nonlocal diffusion and image processing, see \cite{Florin,Gilboa2009}.

The success of nonlocal models stems from the ease with which they handle singularities.  Although classical models using differential equations cannot well describe discontinuous solutions to these equations, such as those occurring in fracture dynamics, the integral operators applied by Silling suffer no such difficulties.  Using nonlocal operators in these classical models extends the types of physical processes that we can model and the types of qualitative behavior that we can describe.

An important feature of these operators is that they are extensions of classical differential operators.  In \cite{DGLZ2013}, it is shown that such first order nonlocal operators converge strongly in $L^2$ to classical partial derivatives as a nonlocality parameter vanishes.  This extensivity property is important since it allows us to preserve much of the physical structure that makes up classical differential models.  In addition, when the classical descriptions are correct, the nonlocal frameworks can recover these results by passing to the ``classical limits".

The majority of the nonlocal literature has focused on so-called second order models.  These nonlocal models extend second order partial differential equations and boundary value problems, such as the wave equation \cite{weckner}, Laplace's equation \cite{DuMengesha} and the heat equation \cite{Rossi} (see also the nonlocal counterpart to the fourth order biharmonic equation \cite{radu15}).  However, the literature for nonlocal extensions of first order models is more sparse.  Du et al \cite{du2012} studied a nonlocal extension of the nonlinear advection equation.  Among other results, they showed that inviscid solutions to this nonlinear equation do not blow up in finite time, a stark contrast to those of the classical equation.  There is also some literature available for first order models using fractional derivatives (see \cite{heymans} for a large set of examples).

Given the success of these nonlocal operators in extending second order models, we are interested in their application to first order models, specifically first order ordinary differential equations (ODEs).  First order models are ubiquitous in the natural sciences.  One area that is currently under active research is the treatment of discontinuous models.  These are ODEs that have discontinuous ``forcing terms", and have applications in the flow through porous soil \cite{flynn}, static friction problems \cite{stewart}, and optimal control with discrete feedbacks \cite{marigo}.  A challenge with classical (i.e. differential) frameworks is to solve these problems numerically \cite{dieci}, since the classical derivatives for solutions to these problems do not exist everywhere.  As a result, many heuristic and complicated numerical methods are needed to solve these equations computationally.  Although we do not pursue this here, one motivation for studying nonlocal first order models is the possibility of replacing classical derivatives in these equations with nonlocal derivatives.  It is possible that solving discontinuous integral equations, which do not contain classical derivatives, is a more straightforward task than solving discontinuous ODEs.  

The contributions of this paper are as follows.

\begin{itemize}

\item {\bf Weak operator convergence:}  We show that the nonlocal derivative converges weakly to the classical derivative.

\item {\bf General solution for antiderivative:}  Using Fourier transforms, we find a general formal expression for the nonlocal antiderivative.  In general, it is infinite-dimensional, but we identify a class of kernel for which it is one-dimensional, analogous to the classical solution.

\item {\bf Weak convergence of solutions:}  Given some hypotheses, we show that the nonlocal antiderivative is a tempered distribution and that it converges weakly to the classical antiderivative.  We also show strong $L^p$ convergence under additional hypotheses.

\end{itemize}

In \cref{sec:der}, we give an overview of the nonlocal derivative operator and prove that it converges weakly to the classical derivative.  Section \ref{sec:int} focuses on nonlocal integration.  We find the general form of the nonlocal antiderivative in \cref{subsec:gen}, provide explicit examples in \cref{subsec:ex}, and discuss the convergence of the antiderivative to the classical solution in \cref{subsec:conv}.  After showing in \cref{subsubsec:hom} that the homogeneous solution converges weakly to the constant function, we prove the inhomogeneous term's convergence in \cref{subsubsec:inh}.  %We construct nonlocal Appell polynomials in \cref{subsec:appell}.

\section{Nonlocal derivative}
\label{sec:der}
The nonlocal derivative operator used here was presented before in \cite{DGLZ2013}.

\begin{defn}[Nonlocal derivative]\label{def:der}
A nonlocal derivative $D_{\alpha,\epsilon}$ is a convolution operator acting on $u:\mathbb{R}\to\mathbb{R}$ with kernel $\alpha_\epsilon$ depending on a nonlocality parameter $\epsilon>0$:
\begin{align}\label{grad}
D_{\alpha,\epsilon}u(t):=-\alpha_\epsilon\ast u(t)=\int_{\mathbb{R}}\alpha_\epsilon(s-t)u(s)\mathrm ds,\hspace{4 mm}t\in\mathbb{R}.
\end{align}
The kernel should satisfy the following conditions:

(i) Scaling and normalization:
\begin{align}\label{kprop:scale}
&\alpha_\epsilon(s)=\sigma_\epsilon\alpha(y/\epsilon),\hspace{4 mm}\sigma_\epsilon=\left[\epsilon^2\int_{\mathbb{R}}s\alpha(s)\mathrm ds\right]^{-1},
\end{align}

(ii) Finite dipole moment:
\begin{align}\label{kprop:int}
&0<\left|\int_{\mathbb{R}}s\alpha(s)\mathrm ds\right|=:|\alpha_{(1)}|<\infty,
\end{align}

(iii) Anti-symmetry:
\begin{align}\label{kprop:anti}
&\alpha(s)=s\,\varphi(|s|).
\end{align}
\end{defn}
\begin{rem}
We will always impose another condition that regulates the behavior of $\alpha$ at infinity.  How restrictive this condition needs to be depends on the application.  Examples: $\text{supp}\,\alpha\subset[-1,1]$, or $\int_{\mathbb{R}}|s^j\alpha(s)|\mathrm ds<\infty$ for each $j=1,2,...$.
\end{rem}
\begin{rem}
Du et al \cite{DGLZ2013} used a weaker form of normalization than the explicit scaling in \eqref{kprop:scale}, but the convenience of this definition may be worth the cost in generality.
\end{rem}

These conditions characterize $D_{\alpha,\epsilon}$ as a nonlocal extension of $d/dt$.  They ensure that $D_{\alpha,\epsilon}u\to u'$ in some sense as $\epsilon\to 0$.  Of course, in which sense this convergence occurs and the space that $D_{\alpha,\epsilon}u$ lives in depend greatly on the space that $u$ lives in.  For example, Du et al \cite{DGLZ2013} proved that, if $\alpha$ has compact support, and $u$ is in $W^1_2$, then $D_{\alpha,\epsilon}u$ is in $L^2$ and that $\|D_{\alpha,\epsilon}u-u'\|_{L^2}\to 0$ as $\epsilon\to 0$.  But even if $u$ is a distribution (i.e. a bounded linear functional on the space $C^\infty_c$), we can still obtain weak convergence.  

\begin{thm}[Weak convergence]\label{thm:weak}
Let $u$ be a distribution and $\alpha$ have compact support.  Then $D_{\alpha,\epsilon}u\rightharpoonup u'$ in the $C^\infty_c$ topology.
\end{thm}
\begin{proof}
We first show strong convergence for any $\psi$ in $C^\infty_c$.  By Taylor's Theorem and \eqref{kprop:scale}:
\begin{align}\label{conv}
\begin{split}
D_{\alpha,\epsilon}\psi(t)&=\int_{\mathbb{R}}\alpha_{\epsilon}(s)\psi(s+t)\mathrm ds=\int_{\mathbb{R}}\alpha_{\epsilon}(s)\left[\psi(t)+s\psi'(t)+\frac{1}{2}s^2\psi''(\theta(s,t))\right]\mathrm ds\\
&\le \psi'(t)+\frac{1}{2}\|\psi''\|_{\infty}\int_{\mathbb{R}}s^2|\alpha_\epsilon(s)|\mathrm ds\\
&=\psi'(t)+\frac{|\alpha|_{(2)}}{2|\alpha_{(1)}|}\,\epsilon\,\|\psi''\|_{\infty},
\end{split}
\end{align}
where $\|.\|_\infty$ denotes the supremum norm, and $\theta(s,t)\in(t,s+t)$ comes from Taylor's theorem.  It is clear from this estimate that $\|D_{\alpha,\epsilon}\psi-\psi'\|_\infty\to 0$ as $\epsilon\to 0$.

Since $\frac{d}{dt}D_{\alpha,\epsilon}\psi=D_{\alpha,\epsilon}\psi'$, and $\psi'$ is in $C^\infty_c$, we see that the above result holds for any derivative of $\psi$.  More precisely, for any $n=0,1,2,...$, we have $\|D_{\alpha,\epsilon}\psi^{(n)}-\psi^{(n+1)}\|_\infty\le \epsilon|\alpha|_{(2)}\|\psi^{(n+2)}\|_\infty/2|\alpha_{(1)}|\to 0$ as $\epsilon\to 0$.

Since $\alpha$ has compact support, $D_{\alpha,\epsilon}\psi$ is in $C^\infty_c$ for each $\epsilon$, with $\text{supp}\,D_{\alpha,\epsilon}\psi\subset \text{supp}\,D_{\alpha,1}$ if $\epsilon\le 1$.  Thus, $D_{\alpha,\epsilon}\psi\to \psi'$ in the (strong) sense of $C^\infty_c$ (for more details on this topology, see e.g. \cite{rudin,strichartz}).

Let $(,)$ be the standard inner product.  Since $u$ is a continuous functional on $C^\infty_c$, we can use the strong convergence to conclude that $(u,D_{\alpha,\epsilon}\psi)\to(u,\psi')$ for any $\psi$ in $C^\infty_c$.  But we define the distributions $D_{\alpha,\epsilon}u$ and $u'$ by the inner products $(D_{\alpha,\epsilon}u,\psi):=(u,D_{\alpha,\epsilon}^*\psi)$ and $(u',\psi):=-(u,\psi')$, where $D_{\alpha,\epsilon}^*$ is the adjoint of $D_{\alpha,\epsilon}$.  In fact, as in the classical case, we have $D_{\alpha,\epsilon}^*=-D_{\alpha,\epsilon}$ \cite{DGLZ2013}.  This means that $(D_{\alpha,\epsilon}u,\psi)\to (u',\psi)$ for every $\psi$, or that $D_{\alpha,\epsilon}u\rightharpoonup u'$.
\end{proof}

\begin{rem}
We assumed that $\alpha_\epsilon$ had compact support in order for $(u,\alpha_\epsilon\ast\psi)$ to make sense, given that $u$ is a general distribution.  Different assumptions about the space that $u$ lives in require different restrictions on the kernel $\alpha$.  But in each case, the proof of weak convergence proceeds almost identically.
\end{rem}

By imposing stronger hypotheses on $u$, it is likely that this weak convergence result can be replaced by something stronger, as Du et al. \cite{DGLZ2013} showed for $W^1_2$ functions $u$.  Figure \ref{fig:gradCon} demonstrates convergence for $u(t)=|t|^{1/2}$ and $\alpha(s)=s\mathds{1}_{(-1,1)}(s)$, where $\mathds{1}_{(-1,1)}$ is the indicator function on $(-1,1)$.  The weak convergence result in Theorem \ref{thm:weak} applies, but the convergence appears to be quite strong.  Indeed, away from $t=0$ (say, on $\mathbb{R}\setminus[-1/2,1/2]$), the convergence of the nonlocal derivatives to the classical derivative appears to be uniform even in the higher derivatives.  The reason for this is that $u$ is a continuous function that is $C^\infty$ almost everywhere, so it is significantly more regular than a distribution.
\begin{figure}[H]
    \includegraphics[width=1\textwidth]{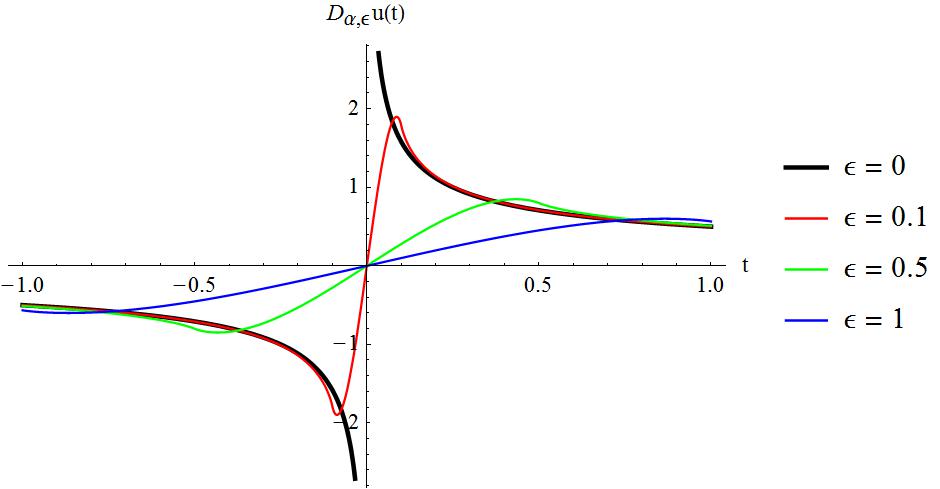}
    \caption{$D_{\alpha,\epsilon}u(t)$ for $u(t)=|t|^{1/2}$ and $\alpha(s)=s\mathds{1}_{(-1,1)}$(s).}
    \label{fig:gradCon}
\end{figure}

We note that the anti-symmetry condition (\ref{kprop:anti}) is critical for interpreting $D_{\alpha,\epsilon}$ as an extension of the classical derivative $d/dt$, or as a ``first order" operator, as evident from its role in (\ref{conv}) of canceling the $u(t)$ term.  Symmetric kernels $\alpha_{\epsilon}$ would instead give something similar to a ``second order" operator like $d^2/dt^2$; indeed, convergence to such for a similar integral operator was proven in \cite{zhou}.

\section{Nonlocal integration}
\label{sec:int}
In the classical setting, the problem $u'(t)=F(t)$ yields, in general, the solution $u(t)=A+\int F(t)\mathrm dt$ (i.e. the ``antiderivative").  We investigate this problem using the nonlocal derivative $D_{\alpha,\epsilon}$ in place of $d/dt$.  More specifically, we consider a given function $F=F(t)$ and an unknown distribution $u_\epsilon$ that satisfy:
\begin{align}\label{prob:int}
D_{\alpha,\epsilon}u_\epsilon(t)=-\alpha_\epsilon\ast u_\epsilon(t)=F(t),\hspace{4 mm}t\in\mathbb{R}.
\end{align}

A similar type of convolution equation has been considered in harmonic analysis for $F\equiv 0$ \cite{delsarte,schwartz}, in which case $u$ is said to be a ``mean-periodic" function.  See also \cite{kahane} for the case of nonzero $F$ when $u(t<0)=0$.  In these studies, the kernel is usually taken to be a measure, so it satisfies a non-negativity requirement.  We work with an anti-symmetric kernel, so our work is more analogous to the study of Hilbert transforms.  The main distinction of our study, however, is the connection of these integral equations with the classical differential equation \eqref{ode} in the limit of $\epsilon\to 0$.

Our results can also be connected to the distribution of zeros of $\hat\alpha_\epsilon$, the Fourier transform of $\alpha_\epsilon$.  For entire analytic functions in the complex plane, there is an extensive theory developed in this area; we cite the seminal work of P\'olya \cite{polya} and the review article by Dimitrov and Rusev \cite{dimitrov}.  We also mention the works by Sedletskii, summarized in \cite{sedletskiibook}, that expand on P\'olya's work and also study the completeness of the exponential basis functions arising from these zeros.

\subsection{General solution}
\label{subsec:gen}
In general, we can find a formal solution to (\ref{prob:int}) by taking its Fourier transform \eqref{fourier}:
\begin{align}\label{prob:intF}
-\hat{\alpha}_\epsilon \hat{u}_\epsilon(\xi)=\hat{F}(\xi),\hspace{4 mm}\xi\in\mathbb{R}.
\end{align}

We can formally solve this equation by division if we make some assumptions about how or if $\hat{\alpha}_\epsilon$ vanishes.  
\begin{defn}\label{def:analytic}
We consider those kernels $\alpha$ that satisfy conditions (ii)-(iii) in Definition \ref{def:der} as well as the following condition on its behavior at infinity:
\begin{align}\label{kprop:analytic}
\int_{\mathbb{R}}|s^j\alpha(s)|\mathrm ds\le A_\alpha\,j!
\end{align}
for some $A_\alpha>0$ and each $j=1,2,...$.
\end{defn}
\begin{rem}
Examples of $\alpha$ that satisfy this strengthened condition include rapidly decreasing functions like $\alpha(s)=\text{sgn}(s)\exp(-|s|^j)$ or those with compact support, such as $\alpha(s)=s\beta(|s|)\mathds{1}_{(-1,1)}(s)$.
\end{rem}

If $\alpha$ satisfies conditions \eqref{kprop:int}, \eqref{kprop:anti}, and \eqref{kprop:analytic}, then $\hat\alpha$ is entirely analytic on the real line.  This means that $\hat{\alpha}_\epsilon$ has (at most) countably many real zeros, a finite number of real zeros in each bounded interval, and zeros of finite multiplicity.  See e.g. \cite{dimitrov} for examples of stronger conditions that give more precise control over these zeros.

Given this assumption on $\alpha$, we find that the general (formal) solution to (\ref{prob:intF}) is:
\begin{align}\label{sol:intF}
\hat{u}_\epsilon(\xi)=-\hat{F}(\xi)/\hat{\alpha}_\epsilon(\xi)+\sum_{\sum_{m=0}^k|\hat\alpha^{(m)}_\epsilon(\xi_{j,\epsilon})|=0}A_{j,k}\delta_{\xi_{j,\epsilon}}^{(k)}(\xi),\hspace{4 mm}\xi\in\mathbb{R},
\end{align}
where $g^{(k)}$ is the $k$th derivative of $g$, $\delta_a(x)=\delta_0(x-a)$ is the Dirac measure centered at $a$, and each $A_{j,k}$ is an arbitrary constant.  The lower sum selects over roots $\xi_{j,\epsilon}$ of $\hat\alpha_\epsilon(\xi)$ with multiplicity $k+1$.

We interpret this expression as a formal infinite series.  Note that, by the anti-symmetry condition (\ref{kprop:anti}), we have that $\hat{\alpha}_\epsilon$ is also anti-symmetric, so the $A_{0,0}\,\delta_0$ term always exists.  We order the zeros like $\dots<\xi_{-n}<\xi_{-n+1}<\dots<\xi_0=0<\dots<\xi_{n-1}<\xi_n<\dots$.  By the anti-symmetry of $\hat\alpha_\epsilon$, we have $\xi_{-j,\epsilon}=-\xi_{j,\epsilon}$.

Applying the inverse Fourier transform to (\ref{sol:intF}) and renaming the arbitrary constants $A_{jk}$ gives the following formal solution to (\ref{prob:int}):
\begin{align}\label{sol:int}
u_\epsilon(t)=-\mathcal{F}^{-1}[\hat{F}/\hat{\alpha}_\epsilon](t)+\sum_{\sum_{m=0}^k|\hat\alpha^{(m)}_\epsilon(\xi_{j,\epsilon})|=0}A_{j,k}t^{k}e^{2\pi i\xi_{j,\epsilon}t}.
\end{align}
Note that the $j=k=0$ term corresponds to $A_{0,0}=$ constant, which also occurs in the classical antiderivative: $u(t)=\int F(t)\mathrm dt+A_{0,0}$.  The first term in \eqref{sol:int} is analogous to the integral of $F$, and is exactly such if $\hat\alpha_\epsilon=-2\pi i\xi$.

\subsection{Examples}
\label{subsec:ex}
We are interested in the nonlocal derivative $D_{\alpha,\epsilon}$ being an extension of the classical derivative $d/dt$, so we would like to know for which kernels $\alpha$ these operators have similar properties.
%the similarities between these operators are strongest.  

The classical antiderivative $u(t)=A+\int F(t)\mathrm dt$ only involves one arbitrary constant, while the nonlocal antiderivative (\ref{sol:int}) involves, in general, an infinite number of such constants.  We observe that requiring $|\hat{\alpha}^{(1)}(0)|>0$ and  $|\hat{\alpha}(\xi)|>0$ when $|\xi|>0$ results in all of the $A_{j,k}$ in (\ref{sol:int}) vanishing except for $A_{0,0}$.  The first property holds by (\ref{kprop:int}).  We can guarantee the second property by requiring that $\alpha(s)$ be a decreasing positive function for $s>0$; for example, $\alpha(s)=\text{sgn}(s)\beta(|s|)$ for any $\beta(|s|)$ decreasing with $|s|$ satisfies this condition; of course, there are many more functions, such as $\alpha(s)=s\exp(-|s|)$, that do not need to be decreasing to satisfy this condition.  For a similar problem on the characterization of the positivity of Fourier transforms, see \cite{giraud}.  

Thus, for these types of $\hat{\alpha}_\epsilon$, (\ref{sol:int}) becomes:
\begin{align*}
u_\epsilon(t)=A-\mathcal{F}^{-1}[\hat F/\hat\alpha_\epsilon](t),
\end{align*}
where $A$ is an arbitrary constant.  This is much more analogous to classical integration.

As an explicit example, consider $\alpha_\epsilon(s)=\text{sgn}(s)\exp(-|s|/\epsilon)/2\epsilon^2$ and $F(t)=1/(1+t^2)$.  The classical antiderivative of this function is obviously $u(t)=A+\arctan(t)$.  On the other hand, for this $\alpha_\epsilon$, its nonlocal antiderivative is $u_\epsilon(t)=u(t)+2\epsilon^2t/(1+t^2)^2$.  We clearly have convergence of $u_\epsilon$ to $u$ as $\epsilon\to 0$, so this nonlocal antiderivative extends the classical one.

For another example, consider $\alpha(s)=\sin(\pi s)\mathds{1}_{(-1,1)}(s)$.  Then the Fourier transform of $\alpha_\epsilon$ is $\hat\alpha_\epsilon(\xi)=i\sin(2\pi\epsilon\xi)/[\epsilon(4\epsilon^2\xi^2-1)]$, which is zero at $\xi=0,\pm 1/\epsilon,\pm 3/2\epsilon,\pm 2/\epsilon,...$.  Indeed, one can directly check with integration that $-\alpha_\epsilon\ast g_n\equiv 0$ for $g_n(t)=\exp(n\pi i t/\epsilon)$ and $n=0,\pm 2,\pm 3,...$, which agrees with (\ref{sol:int}).  So, in addition to annihilating constants, the nonlocal derivative $D_{\alpha,\epsilon}$ for this kernel actually annihilates an infinite dimensional set of functions.  This is a stark contrast to the classical derivative.  Note, however, that as $\epsilon\to 0$, each nonzero $\xi_{j,\epsilon}$ tends to infinity, so, in a weak sense, the nonlocal antiderivative still recovers the classical antiderivative.

\subsection{Convergence to classical solution}
\label{subsec:conv}
We demonstrate that \eqref{sol:int} converges weakly to the classical solution of $u'=F$, namely $u=A+\int F$.  There are two types of terms in \eqref{sol:int} that we need to consider: the homogeneous terms (summation) from the equation $D_{\alpha,\epsilon}u_\epsilon=0$ and the inhomogeneous term that involves $F$.

\subsubsection{Homogeneous part}
\label{subsubsec:hom}
We first show that the zero set of $\hat\alpha_\epsilon$ spreads away from the origin as $\epsilon\to 0$.  The scaling assumption in \eqref{kprop:scale} is critical here.
\begin{lem}[Zero Set]\label{lem:zero}
Let $\alpha_\epsilon$ satisfy the conditions of Definitions \ref{def:der} and \ref{def:analytic}.  Let $\xi_{j,\epsilon}$ denote the $j$th zero of $\hat\alpha_\epsilon$ as in \eqref{sol:intF}.  Then, for each $j=1,2,3,...$, we have $\xi_{\pm j}\to\pm\infty$ as $\epsilon\to 0$.
\end{lem}
\begin{proof}
Since $\xi_{-j,\epsilon}=-\xi_{j,\epsilon}$, it suffices to show that $\xi_{j,\epsilon}\to\infty$ for each $j=1,2,3,...$.  At a zero $\xi_{j,\epsilon}$, we have:
\begin{align*}
\hat\alpha_\epsilon(\xi_{j,\epsilon})=0=-i\int_{-\infty}^\infty\sin(2\pi\xi_{j,\epsilon} t)\alpha_\epsilon(t)\mathrm dt.
\end{align*}
Using the scaling in \eqref{kprop:scale}, we can rewrite this after letting $t\to \epsilon s$ as:
\begin{align}\label{alphahat}
0=-i\int_{-\infty}^\infty\sin(2\pi\epsilon\xi_{j,\epsilon} s)\alpha(s)\mathrm ds=\hat\alpha(\epsilon\xi_{j,\epsilon})=:\hat\alpha(\bar\xi_j),
\end{align}
where $\bar\xi_j:=\epsilon\xi_{j,\epsilon}$ is the $j$th root of $\hat\alpha$ and does not depend on $\epsilon$.  Thus, $\xi_{j,\epsilon}=\bar\xi_j/\epsilon\to\infty$ as $\epsilon\to 0$.
\end{proof}

We can now show that the homogeneous part of \eqref{sol:int} converges weakly to the classical solution (i.e. the constant function) as $\epsilon\to 0$.  We first consider finite sums in \eqref{sol:int}.
\begin{thm}[Finite sums]\label{thm:conv:int:fin}
Let $u_\epsilon$ be the solution of \eqref{prob:int} given by \eqref{sol:int} for $F=0$, and 
%let $\alpha_\epsilon$ supported in $[-\epsilon,\epsilon]$ satisfy conditions \eqref{kprop:scale}-\eqref{kprop:anti} and \eqref{kprop:pos}
let $\alpha_\epsilon$ satisfy the conditions of Definitions \ref{def:der} and \ref{def:analytic}.  Suppose that the formal sum in \eqref{sol:int} is finite.  Then $u_\epsilon\rightharpoonup A_{0,0}$ as $\epsilon\to 0$.
\end{thm}
\begin{proof}
Let $\psi$ be an arbitrary function in $C^\infty_c$.  For any $j:|j|>1$ and any $k\ge 0$, the Riemann-Lebesgue lemma and Lemma \ref{lem:zero} show that $(t^ke^{2\pi i\xi_{j,\epsilon}t},\psi)\to 0$ as $\epsilon\to 0$.  Since $\psi$ is arbitrary, it follows that each nonconstant term in the formal sum \eqref{sol:int} converges weakly to zero.  The result follows since the sum is finite.
\end{proof}

We can extend this to infinite series if we restrict the formal sum in \eqref{sol:int} to converge in a suitable sense.  It is convenient to map the terms of the series to the sequence $\{U_{n,\epsilon}\}$ and to consider the partial sums $u_{N,\epsilon}=\sum_{n=1}^NU_{n,\epsilon}$.  To interchange limit operations, we also make use of the counting measure $d\mu$: $\int f_nd\mu_n:=\sum_{n=1}^\infty f_n$.
\begin{thm}[Infinite series]\label{thm:conv:int}
Let $u_\epsilon$ and $\alpha_\epsilon$ be as in Theorem \ref{thm:conv:int:fin}.  Suppose that the following hypotheses hold:  

(i)  For each $\epsilon>0$, $u_{N,\epsilon}(t)\to u_{\epsilon}(t)$ a.e. in $\mathbb{R}$ as $N\to\infty$.

(ii)  For each $\epsilon$, the sequence $\{u_{N,\epsilon}\}_{N=1}^\infty$ of partial sums is uniformly locally integrable.  That is, for every $n$ and each $E\subset\mathbb{R}$ with $|E|<\infty$, we have $\left|\int_{E}u_{N,\epsilon}(t)\mathrm dt\right|\le \mathcal{U}_\epsilon(|E|)$, where $\mathcal{U}_\epsilon>0$ is locally bounded, and $\mathcal{U}_\epsilon(x)\to 0$ as $x\to 0$.

(iii)  For any $\psi$ in $C^\infty_c$, the set of integrals $\{(U_{n,\epsilon},\psi)\}_{\epsilon>0}$ is uniformly bounded for almost every $n$ by a summable sequence $\{v_{n,\psi}\}\subset\mathbb{R}_+$.  That is, for each $\epsilon>0$ and each $\psi$ in $C^\infty_c$, we have $v_{n,\psi}\ge |(U_{n,\epsilon},\psi)|$ for every $n=1,2,3,...$ except for a finite set, and $\sum_{n=1}^\infty v_{n,\psi}<\infty$.

Then $u_\epsilon\rightharpoonup A_{0,0}$ as $\epsilon\to 0$.
\end{thm}
\begin{proof}
Fix an arbitrary $\psi$ in $C^\infty_c$.  Since $|\text{supp}\,\psi|<\infty$ and (ii) holds, Vitali's convergence theorem applies to $\{u_{N,\epsilon}\}_{N=1}^\infty$.  We compute:
\begin{align*}
\lim_{\epsilon\to 0}(u_\epsilon,\psi)=\lim_{\epsilon\to 0}\left(\lim_{N\to\infty}u_{N,\epsilon},\psi\right)=\lim_{\epsilon\to 0}\lim_{N\to\infty}\sum_{n=1}^N(U_{n,\epsilon},\psi)=\lim_{\epsilon\to 0}\int(U_{n,\epsilon},\psi)d\mu_n.
\end{align*}
By (iii), the dominated convergence theorem applies to the last integral, so by Theorem \ref{thm:conv:int:fin}, we obtain:
\begin{align*}
\lim_{\epsilon\to 0}(u_\epsilon,\psi)=\int\lim_{\epsilon\to 0}(U_{n,\epsilon},\psi)d\mu_n=\int 0\,d\mu_n=0.
\end{align*}
\end{proof}

\begin{rem}
If we allow the coefficients $A_{j,k}$ to vanish with $\epsilon$, then this weak convergence can be made strong (say, in $L^\infty_{loc}$).
\end{rem}

\begin{rem}
Condition (ii) is satisfied if $u_{N,\epsilon}\to u_\epsilon$ in $L^1_{loc}$.  Condition (iii) is satisfied if the zeros of $\hat\alpha_\epsilon$ have a maximum multiplicity of $M\ge 1$ and if $\sum_{j,k}|A_{j,k}|<\infty$, where the coefficients $A_{j,k}$ are as in \eqref{sol:int}.  Indeed, since $U_{n,\epsilon}(t)=A_nt^{k_n}e^{2\pi i\xi_{j_n}t}$ for some $j_n$, $k_n$, and $A_n=:A_{j_n,k_n}$, we have $|U_{n,\epsilon}(t)|\le |A_n||t|^{M}$ for each $\epsilon$.  This gives:
\begin{align*}
\sum_n|(U_{n,\epsilon},\psi)|\le\sum_n|A_n|\int_{\text{supp}\,\psi}\bigl |t^{M}\psi(t)\bigr|\mathrm dt\le 2T_\psi^{M+1}\|\psi\|_\infty\sum_{n}|A_n|, 
\end{align*}
where $T_\psi=\max(|\inf\text{supp}\,\psi|,|\sup\text{supp}\,\psi|)$.
\end{rem}

\subsubsection{Inhomogeneous part}
\label{subsubsec:inh}

We now investigate the convergence of the inhomogeneous term in \eqref{sol:int} that depends on the ``derivative" $F$.
\begin{defn}
We denote the inhomogeneous part of $u_\epsilon$ in \eqref{sol:int} as follows:
\begin{align}\label{vdef}
\begin{split}
v_\epsilon(t)=-\mathcal{F}^{-1}[\hat F/\hat\alpha_\epsilon](t).
\end{split}
\end{align}
\end{defn}
We first show under what conditions this formal solution $v_\epsilon$ is well-defined.  Then, provided that it is well-defined, we show that $v_\epsilon\rightharpoonup v:=\int F$, or that it converges to a classical antiderivative in the weak sense of tempered distributions.

A necessary condition for $v_\epsilon$ to be well-defined is for $\hat F$ to make sense.  We recall the definition of a tempered distribution.
\begin{defn}\label{def:temp}
Let $\mathcal{S}$ be the Schwartz space of infinitely differentiable rapidly decreasing test functions.  Equivalently, $\psi$ is in $\mathcal{S}$ if it is in $C^\infty$ and if $\lim_{|t|\to\infty}|t|^n\psi(t)=0$ for every integer $n\ge 0$.  The space of tempered distributions $\mathcal{S}'$ is the continuous dual to $\mathcal{S}$ (i.e. the set of bounded linear functionals on $\mathcal{S}$).
\end{defn}

\begin{defn}\label{def:ftemp}
We require that $F$ in \eqref{prob:int} and \eqref{sol:int} be a tempered distribution, and we allow it to be a generalized function.  
\end{defn}

\begin{rem}
Unfortunately, this is a very strong condition.  In the classical case, we could find an antiderivative of \emph{any} locally integrable function $F$ (say, $F(t)=e^{t}$).  A similar problem concerns solutions to $\Delta u=0$; there are many solutions to this equation, such as $e^{x_1}\cos x_2$, that are not tempered distributions, and Fourier transform techniques cannot recover these solutions.
\end{rem}

\begin{rem}
Note that, from this restriction and Definition \ref{def:analytic}, the distributional product $\hat F/\hat\alpha_\epsilon$ always makes sense, since we have that $2\pi i/\hat\alpha_\epsilon-1/\xi$ is a $C^\infty$ function.  Of course, we still need to show that it makes sense as a \emph{tempered} distribution.
\end{rem}

We now wish for the inverse Fourier transform in \eqref{vdef} to make sense.  There are two problems with this.  The first problem concerns the possibly infinite number of singularities of $1/\hat\alpha_\epsilon$, and the second problem comes from the behavior of $1/\hat\alpha_\epsilon$ at infinity.  If each of these behaviors is not sufficiently controlled, then $\hat F/\hat\alpha_\epsilon$ fails to be a tempered distribution, and \eqref{vdef} makes no sense. 

To solve the first problem, we could impose that $1/\hat\alpha$ have a finite number of singularities.  However, for simplicity, we will restrict $1/\hat\alpha$ to be singular only at the origin.  The more general case proceeds similarly.  Note that this means we no longer need the full strength of Definition \ref{def:analytic}, but we keep it for simplicity. 
\begin{defn}\label{def:alphapos}
We require that $\hat\alpha(\xi)=0$ for $\xi$ in $\mathbb{R}$ only if $\xi=0$.  A sufficient condition for this (in addition to those in Definitions \ref{def:der} and \ref{def:analytic}) is as follows.  We require that $\alpha(s)=\text{sgn}(s)\beta(|s|)$, where $\beta\ge 0$ is a decreasing function that is strictly decreasing a.e. in $(0,a_\alpha)$ for some $a_\alpha>0$.
\end{defn}
\begin{rem}
The condition is sufficient by the alternating series test.  Indeed, the Fourier integral $\int_0^\infty\sin(2\pi\xi s)\alpha_\epsilon(s)\mathrm ds=\sum_{n=0}^\infty(-1)^n\int_0^{1/2\xi}\sin(2\pi\xi s)\alpha_\epsilon(s+n/2\xi)\mathrm ds>0$ if $\xi>0$ and if this condition is satisfied.  Examples of admissible $\beta$ include $(1-s^m)\mathds{1}_{(0,1)}(s)$ for $m>0$, $\exp(-|s|^m)$ for $m>0$, and $|s|^{-m}\mathds{1}_{(0,1)}$ for $-2<m<0$.
\end{rem}

As it turns out, this condition almost solves the second problem as well.  We only need two modifications to it.  We gauge the behavior of $\alpha$ near zero using power functions, and we impose that, near zero, $\alpha$ cannot be flatter than a continuous power function.
\begin{defn}\label{def:flat}
There are two cases to consider, depending on how singular $\alpha$ is at zero.  Let $\alpha_{0+}=\lim_{s\searrow 0}\alpha(s)$.  In all cases, we presume the existence of some endpoint $b_\alpha:0<b_\alpha\le a_\alpha$ (cf. Definition \ref{def:alphapos}), some exponent $k_{\alpha}$ that satisfies $-2<k_{\alpha}\neq 0$, and some constants $K_{\pm,\alpha}>0$.

(i)  If $\alpha_{0+}<\infty$, then we require that $-K_{-,\alpha}s^{k_{\alpha}}\le \alpha(s)-\alpha_{0+}\le -K_{+,\alpha}s^{k_{\alpha}}$ for each $s$ in $(0,b_\alpha)$.  Here, $0<k_{\alpha}$, and $K_{-,\alpha}\ge K_{+,\alpha}$.  We assume that $b_\alpha$ can be made small enough so that we can choose $K_{\pm,\alpha}$ that satisfy $K_{+,\alpha}(s+1)^{k_{\alpha}}-K_{-,\alpha}s^{k_{\alpha}}\ge 0$ for each s on $(0,1)$.

(ii)  If $\alpha_{0+}=\infty$, then we require that $K_{-,\alpha}s^{k_{\alpha}}\le \alpha(s)\le K_{+,\alpha}s^{k_{\alpha}}$ on $(0,b_\alpha)$.  Here, $-2<k_{\alpha}<0$, and $K_{-,\alpha}\le K_{+,\alpha}$.  We again assume that, by choosing $b_\alpha$ small enough, we can find $K_{\pm,\alpha}$ such that $K_{-,\alpha}s^{k_\alpha}-K_{+,\alpha}(s+1)^{k_\alpha}\ge 0$ on $(0,1)$.
\end{defn}

\begin{rem}
``Flat functions" like $\alpha(s)=\text{sgn}(s)(1-\exp(-|s|^{-2}))\mathds{1}_{(0,1)}(|s|)$ are not admissible.
\end{rem}

\begin{rem}
By anti-symmetry, analogous inequalities hold for $s<0$.  Although power functions are not the most general gauge functions (e.g. logarithms are also natural choices), they are quite convenient here due to their homogeneity under dilations.
\end{rem}

\begin{rem}
Since $\alpha$ satisfies both the anti-symmetry and positivity conditions of Definitions \ref{def:der} and \ref{def:alphapos}, we see that it cannot be continuous at zero.  It either approaches a constant or blows up as $s\searrow0$.  This means that $\alpha_{0+}>0$.
\end{rem}

To prove that $\hat v_\epsilon$ in \eqref{vdef} is a tempered distribution, given these restrictions on $\alpha$, we need the following two estimates on $\alpha_\epsilon$.  They elucidate its behavior for small and large $|\xi|$.

\begin{lem}[Near-field]\label{lem:est}
Let $\alpha_\epsilon$ satisfy the conditions of Definitions \ref{def:der}, \ref{def:analytic}, and \ref{def:alphapos}.  Then the following estimate holds for each $\xi$ in $\mathbb{R}$:
\begin{align}\label{est}
2\pi|\xi|-C_\alpha\epsilon^2|\xi|^3\le i\hat\alpha_\epsilon(|\xi|)\le 2\pi|\xi|,
\end{align}
where $C_\alpha=4\pi^3\alpha_{(3)}/3\alpha_{(1)}$, and $\alpha_{(j)}=\int_{-\infty}^\infty s^j\alpha(s)\mathrm ds$.
\end{lem}
\begin{proof}
Using the anti-symmetry of $\alpha$, we can write the Fourier transform \eqref{fourier} of $\alpha_\epsilon$ as a Fourier-sine transform:
\begin{align*}
-\hat\alpha_\epsilon(\xi)/2i=\int_{0}^\infty\sin(2\pi\xi t)\alpha_\epsilon(t)\mathrm dt.
\end{align*}
Using the scaling in \eqref{kprop:scale}, we can rewrite this after letting $t\to \epsilon s$ as:
\begin{align}\label{alphahat0}
-\hat\alpha_\epsilon(\xi)/2i=\frac{1}{\epsilon\,\alpha_{(1)}}\int_{0}^\infty\sin(2\pi\epsilon\xi s)\alpha(s)\mathrm ds.
\end{align}
Suppose first that $\xi>0$.  Using that $\sin(2\pi\epsilon\xi s)\alpha(s)\le 2\pi\epsilon\xi s\alpha(s)$ in \eqref{alphahat} (note the positivity from Definition \ref{def:alphapos}) and that $\int_0^\infty s\alpha(s)\mathrm ds=\alpha_{(1)}/2$ gives:
\begin{align*}
i\hat\alpha_\epsilon(\xi)\le 2\pi\xi.
\end{align*}
On the other hand, using Taylor's theorem, letting $\theta(s)$ be in $(0,\infty)$, and applying the estimate $\cos(\theta(s))\le 1$ to \eqref{alphahat}, we obtain:
\begin{align*}
i\hat\alpha_\epsilon(\xi)/2&=\pi \xi-\frac{4\pi^3\xi^3\epsilon^2}{3\alpha_{(1)}}\int_{0}^\infty\cos(\theta(s))s^3\alpha(s)\mathrm ds\\
&\ge \pi \xi-2\pi^3\alpha_{(3)}\xi^3\epsilon^2/3\alpha_{(1)}.
\end{align*}
If $\xi<0$, then we reverse the above inequalities.  This gives \eqref{est}.
\end{proof}

\begin{lem}[Far-field]\label{lem:est:alt}
Let $\alpha_\epsilon$ satisfy the conditions of Definitions \ref{def:der}, \ref{def:analytic}, \ref{def:alphapos}, and \ref{def:flat}.  Then, for each $\epsilon>0$ and each $\xi:|\xi|\ge 2\epsilon b_\alpha$ (cf. Definition \ref{def:flat}), the following estimate holds:
\begin{align}\label{est:alt}
i\hat\alpha_\epsilon(|\xi|)\ge C'_{\alpha}\epsilon^{-2-k_\alpha}|\xi|^{-1-k_\alpha},
\end{align}
where $k_\alpha$ is as in Definition \ref{def:flat}, and $C'_\alpha>0$. 
\end{lem}

\begin{proof}
Suppose that $\xi>0$ (by anti-symmetry, we can reverse the resulting inequality when $\xi<0$).  We use the positivity in Definition \ref{def:alphapos} and the scaling in Definition \ref{def:der} to obtain the following alternating series-type estimate:
\begin{align}\label{est:alt0}
\begin{split}
i\hat\alpha_\epsilon(\xi)/2&=\int_0^\infty\sin(2\pi\xi s)\alpha_\epsilon(s)\mathrm ds\\
&\ge\int_0^{1/\xi}\sin(2\pi\xi s)\alpha_\epsilon(s)\mathrm ds\\
&=\frac{1}{\alpha_{(1)}\epsilon^2\xi}\int_0^1\sin(2\pi s)\alpha(s/\epsilon\xi)\mathrm ds\\
&=\frac{1}{2\alpha_{(1)}\epsilon^2\xi}\int_0^{1}\sin(\pi s)[\alpha(s/2\epsilon\xi)-\alpha((s+1)/2\epsilon/\xi)]\mathrm ds.
\end{split}
\end{align}
We estimate each term using the inequalities in Definition \ref{def:flat}.  To do this, we suppose that $\xi\ge 1/2\epsilon b_\alpha$ (cf. Definition \ref{def:flat}).  For each case, \eqref{est:alt0} becomes:

Case (i):
\begin{align}\nonumber
\begin{split}
i\hat\alpha_\epsilon(\xi)&\ge \frac{1}{2^{k_\alpha}\alpha_{(1)}\epsilon^{2+k_\alpha}\xi^{1+k_\alpha}}\int_0^1\sin(\pi s)[K_{+\alpha}(s+1)^{k_\alpha}-K_{-\alpha}s^{k_\alpha}]\mathrm ds\\
&:=\frac{C'^{(i)}_{\alpha}}{\epsilon^{2+k_\alpha}\xi^{1+k_\alpha}}.
\end{split}
\end{align}

Case (ii):
\begin{align}\nonumber
\begin{split}
i\hat\alpha_\epsilon(\xi)&\ge \frac{1}{2^{k_\alpha}\alpha_{(1)}\epsilon^{2+k_\alpha}\xi^{1+k_\alpha}}\int_0^1\sin(\pi s)[K_{-,\alpha}s^{k_\alpha}-K_{+,\alpha}(s+1)^{k_\alpha}]\mathrm ds\\
&:=\frac{C'^{(ii)}_{\alpha}}{\epsilon^{2+k_\alpha}\xi^{1+k_\alpha}}.
\end{split}
\end{align}
\end{proof}

\begin{thm}\label{thm:temp}
Let $F$ satisfy Definition \ref{def:ftemp} and $\alpha_\epsilon$ Definitions \ref{def:der}, \ref{def:analytic}, \ref{def:alphapos}, and \ref{def:flat}.  For each $\epsilon>0$, the generalized function $1/\hat\alpha_\epsilon$ is a tempered distribution.  In addition, $\hat v_\epsilon=-\hat F/\hat\alpha_\epsilon$ is a tempered distribution.
\end{thm}

\begin{proof}
We let $1/\hat\alpha_\epsilon(\xi)=-1/2\pi i\xi-\hat\beta_\epsilon(\xi)/2\pi i$, where, for $|\xi|$ sufficiently small, $\hat\beta_\epsilon$ satisfies the following estimate (cf. \eqref{est}):
\begin{align}\label{est:beta:near}
|\hat\beta_\epsilon(\xi)|=\frac{|2\pi\xi-i\hat\alpha_\epsilon(\xi)|}{|\xi\hat\alpha_\epsilon(\xi)|}\le \frac{C_\alpha\epsilon^2|\xi|^2}{|\hat\alpha_\epsilon(\xi)|}\le\frac{C_\alpha\epsilon^2|\xi|}{(2\pi-C_\alpha\epsilon^2|\xi|^2)}.
\end{align}
This shows that $\hat\beta_\epsilon$ is continuous at $\xi=0$.  Since $\hat\alpha_\epsilon$ has no other zeros (cf. Definition \ref{def:alphapos}), we conclude that $\hat\beta_\epsilon$ is a locally bounded function on $\mathbb{R}$ (i.e. it has no singularities).  

We claim that $\hat\beta_\epsilon$ is a tempered distribution.  Since it has no singularities, we need only examine its asymptotic behavior at infinity.  Using Lemma \ref{lem:est:alt}, in terms of $\hat\beta_\epsilon(\xi)=-1/\xi+2\pi/i\hat\alpha_\epsilon(\xi)$, we obtain the following inequality for $|\xi|\ge 1/2\epsilon b_\alpha$:
\begin{align}\label{est:beta:far}
\hat\beta_\epsilon(|\xi|)|\le\frac{2\pi\epsilon^{2+k_\alpha}}{C'_\alpha}|\xi|^{1+k_\alpha}-\frac{1}{|\xi|}.
\end{align}
Since this holds for each $\epsilon>0$, we see that the function $\hat\beta_\epsilon$ is a tempered distribution.  Since $1/\xi$ (interpreted in the principal value sense) is a tempered distribution, this means that $1/\hat\alpha_\epsilon$ is also a tempered distribution.

To show that $\hat v_\epsilon$ is a tempered distribution, we write the product $\hat v_\epsilon$ as follows:
\begin{align*}
2\pi i\hat v_\epsilon=-2\pi i\hat F/\hat\alpha_\epsilon=\hat F/\xi+\hat F\hat\beta_\epsilon.
\end{align*}
If we interpret $1/\xi$ in the principal value sense, then the first term corresponds to the distributional integral of $F$ (see e.g. Problem 6.9 in \cite{strichartz}).  The second term is also a tempered distribution, since $(\hat F_s\hat\beta_\epsilon,\psi)=(\hat F,\hat\beta_\epsilon\psi)=:(\hat F,\psi_\beta)$, where, for an arbitrary test function $\psi$, $\psi_\beta$ is also a test function, since $\hat\beta_\epsilon$ is both tempered and $C^\infty$ (cf. Definition \ref{def:analytic}).  Since $\hat F$ is a tempered distribution, this term makes sense.  Thus, $\hat v_\epsilon$ is well-defined.
\end{proof}

Unlike in the classical case, the nonlocal antiderivative $v_\epsilon$ can be much rougher than its ``derivative" $F$.  This occurs if $\alpha_\epsilon$ is chosen to be too flat at zero.  For example, in the estimate \eqref{est:alt}, if $\alpha_\epsilon$ is Lipschitz continuous at zero (i.e. $k_\alpha=1$), then we see that $i\hat\alpha_\epsilon(|\xi|)\ge C_{\alpha,\epsilon}/\xi^2$ for $\xi$ sufficiently large.  This means that $\hat F(\xi)/\hat\alpha_\epsilon(\xi)\sim \hat F(\xi)\xi^2$ for $\xi$ sufficiently large, which shows that $v_\epsilon$, in this case, is as smooth as $F''$.  This gets worse the larger that $k_\alpha$ is, or the flatter that $\alpha_\epsilon$ is at zero.  

However, if $\alpha_\epsilon$ is chosen to be sufficiently singular at zero, then nonlocal integration can be a smoothing process, as in the classical case.  From \eqref{est:alt}, we see that if $k_\alpha=-1-\delta$ for $0<\delta<1$, then $i\hat\alpha_\epsilon(|\xi|)\ge C_{\alpha,\epsilon}|\xi|^{-\delta}$.  This means that $\hat v_\epsilon\sim \hat F/|\xi|^{\delta}$ for $|\xi|$ sufficiently large.  Therefore, $v_\epsilon$ is actually smoother than its ``derivative" $F$, for these kernels $\alpha_\epsilon$.

Now that we have made sense of the formal solution \eqref{vdef}, we can address its convergence to the classical solution $v(t)=\int F(t)\mathrm dt$.

\begin{thm}[Weak convergence]\label{thm:conv}
Let $F$ satisfy Definition \ref{def:ftemp} and $\alpha_\epsilon$ Definitions \ref{def:der}, \ref{def:analytic}, \ref{def:alphapos}, and \ref{def:flat}.  Then $v_\epsilon\rightharpoonup v$ in $\mathcal{S}'$ as $\epsilon\to 0$.  That is, $(v_\epsilon,\psi)\to (v,\psi)$ for every $\psi$ in $\mathcal{S}$.
\end{thm}

\begin{proof}
We decompose $2\pi/i\hat\alpha_\epsilon-1/\xi=\hat\beta_\epsilon$ into near- and far-field contributions:
\begin{align*}
\hat\beta_\epsilon(\xi)=\hat\beta_\epsilon(\xi)\mathds{1}_{(0,1/2\epsilon b_\alpha)}(|\xi|)+\hat\beta_\epsilon(\xi)\mathds{1}_{(1/2\epsilon b_\alpha,\infty)}(|\xi|),
\end{align*}
where $b_\alpha$ is as in Definition \ref{def:flat}.  Applying the near- and far-field estimates \eqref{est:beta:near} and \eqref{est:beta:far}, we obtain:
\begin{align}\label{est:beta}
\hat\beta_\epsilon(|\xi|)\le \frac{C_\alpha\epsilon}{2\pi b_\alpha-C_\alpha/4b_\alpha}\mathds{1}_{(0,1/2\epsilon b_\alpha)}(|\xi|)+\left(\frac{2\pi}{C'_\alpha}\epsilon^{2+k_a}|\xi|^{1+k_a}+2\epsilon b_\alpha\right)\mathds{1}_{(1/2\epsilon b_\alpha,\infty)}(|\xi|).
\end{align}
If $\hat\beta_\epsilon$ represents the nonlocal correction to $1/\xi$, then we see that it vanishes pointwise as $\epsilon\to 0$.

We write $2\pi i\hat v_\epsilon$ as $\hat F/\xi+\hat F\hat\beta_\epsilon$ and show that the second term vanishes weakly.  For any test function $\psi$, we let $\check{\psi}:=\mathcal{F}^{-1}\psi$.  We have $(\hat F\hat\beta_\epsilon,\check\psi)=(\hat F,\hat\beta_\epsilon\check\psi)=:(\hat F,\psi_{\epsilon})$, where the test function $\psi_\epsilon(\xi)\to 0$ uniformly as $\epsilon\to 0$.  This shows that $(\hat v_\epsilon,\check\psi)=(v_\epsilon,\psi)\to(v,\psi)$ as $\epsilon\to 0$.
\end{proof}

\begin{rem}
Obviously, we are interested in those kernels $\alpha$ for which the denominator of the first term in \eqref{est:beta} is positive.  Such kernels exist.  For example, if we assume that $\text{supp}\,\alpha$ is in $(-b_\alpha,b_\alpha)$, then, under the hypotheses of Definition \ref{def:flat}, we can guarantee positivity if $-1.402984...<k_\alpha<-0.509742...$.  More specifically, if $K_{-,\alpha}=K_{+,\alpha}$ as well (so that $\alpha$ is a power function), then this holds if $-2<k_{\alpha}<0$ or if $k_\alpha\ge(24-2\pi^2)/(\pi^2-6)=1.1011...$.
%It suffices to have $6b_\alpha^2\alpha_{(1)}>\pi^2\alpha_{(3)}$.  
\end{rem}

If we make $\alpha_\epsilon$ more singular and restrict $F$ to be a function in an $L^p$ space, then we can use the estimate \eqref{est:beta} to obtain stronger convergence results.
\begin{cor}[Strong convergence]
Suppose that $-2<k_\alpha<-1$, such that $\alpha$ is not integrable at zero.  Suppose also that $F$ is in $L^p$, where $1\le p\le\infty$.  Then $v_\epsilon-v$ is in $L^p$, and $v_\epsilon-v\to 0$ in $L^p$.
\end{cor}

\begin{proof}
We have $|\xi|^{1+k_\alpha}\le (2b_\alpha\epsilon)^{|k_\alpha|-1}$ in \eqref{est:beta}, which shows that $\hat\beta_\epsilon(|\xi|)\to 0$ uniformly as $\epsilon\to 0$.  It also shows that $2\pi i(v_\epsilon-v)=\mathcal{F}^{-1}[\hat F\hat\beta_\epsilon]$ is in $L^p$.  Thus, $\hat F\hat\beta_\epsilon\to 0$ in $L^q$, where $1/q=1-1/p$.  This shows that $v_\epsilon-v\to 0$ in $L^p$.
\end{proof}

\section*{Acknowledgments}
The author thanks Prof. Petronela Radu for useful discussions.

\bibliographystyle{abbrv} % abbrv to abbreviate first names
\bibliography{NonlocalREU15}

\end{document}